\documentclass[12pt, a4paper, reqno]{amsart}

\usepackage{amsmath}
\usepackage{amsfonts}
\usepackage{amssymb}
\usepackage{amsthm}
\usepackage{url}
\usepackage{enumerate}
\usepackage[latin1]{inputenc}

\newtheorem{theorem}[equation]{Theorem}

\newtheorem{lemma}[equation]{Lemma}

\newtheorem{definition}[equation]{Definition}

\theoremstyle{remark}

\newtheorem{remark}[equation]{Remark}

\numberwithin{equation}{section}

\def\vint_#1{\mathchoice%
          {\mathop{\kern 0.2em\vrule width 0.6em height 0.69678ex depth -0.58065ex
                  \kern -0.8em \intop}\nolimits_{\kern -0.4em#1}}%
          {\mathop{\kern 0.1em\vrule width 0.5em height 0.69678ex depth -0.60387ex
                  \kern -0.6em \intop}\nolimits_{#1}}%
          {\mathop{\kern 0.1em\vrule width 0.5em height 0.69678ex
              depth -0.60387ex
                  \kern -0.6em \intop}\nolimits_{#1}}%
          {\mathop{\kern 0.1em\vrule width 0.5em height 0.69678ex depth -0.60387ex
                  \kern -0.6em \intop}\nolimits_{#1}}}
\def\vintslides_#1{\mathchoice%
          {\mathop{\kern 0.1em\vrule width 0.5em height 0.697ex depth -0.581ex
                  \kern -0.6em \intop}\nolimits_{\kern -0.4em#1}}%
          {\mathop{\kern 0.1em\vrule width 0.3em height 0.697ex depth -0.604ex
                  \kern -0.4em \intop}\nolimits_{#1}}%
          {\mathop{\kern 0.1em\vrule width 0.3em height 0.697ex depth -0.604ex
                  \kern -0.4em \intop}\nolimits_{#1}}%
          {\mathop{\kern 0.1em\vrule width 0.3em height 0.697ex depth -0.604ex
                  \kern -0.4em \intop}\nolimits_{#1}}}

\newcommand{\aveint}[2]{\mathchoice%
          {\mathop{\kern 0.2em\vrule width 0.6em height 0.69678ex depth -0.58065ex
                  \kern -0.8em \intop}\nolimits_{\kern -0.45em#1}^{#2}}%
          {\mathop{\kern 0.1em\vrule width 0.5em height 0.69678ex depth -0.60387ex
                  \kern -0.6em \intop}\nolimits_{#1}^{#2}}%
          {\mathop{\kern 0.1em\vrule width 0.5em height 0.69678ex depth -0.60387ex
                  \kern -0.6em \intop}\nolimits_{#1}^{#2}}%
          {\mathop{\kern 0.1em\vrule width 0.5em height 0.69678ex depth -0.60387ex
                  \kern -0.6em \intop}\nolimits_{#1}^{#2}}}

\newcommand{\vp}{\varphi}
\newcommand{\eps}{\varepsilon}

\newcommand{\R}{\mathbb{R}}
\newcommand{\Rn}{\mathbb{R}^n}

\newcommand{\dist}{\operatorname{dist}}
\newcommand{\norm}[1]{\left|\left| #1 \right|\right|}
\newcommand{\abs}[1]{\left| #1 \right|}

\newcommand{\inprod}[1]{\langle #1\rangle}

\renewcommand{\limsup}{\operatornamewithlimits{lim \, sup}}
\renewcommand{\liminf}{\operatornamewithlimits{lim \, inf}}
\newcommand{\esssup}{\operatornamewithlimits{ess\, sup}}

\newcommand{\spt}{\operatorname{spt}}

\newcommand{\divt}{\operatorname{div}}

\newcommand{\parts}[2]{\frac{\partial {#1}}{\partial {#2}}}
\newcommand{\pdo}[2]{\frac{\partial #1}{\partial #2}}

\newcommand{\trm}{\textrm}
\newcommand{\beq}{\begin{equation}}
\newcommand{\eeq}{\end{equation}}

\newcommand{\ud}{\, d}
\newcommand{\Om}{\Omega}

\newcommand{\ol}{\overline}
\newcommand{\B}{\ensuremath{\mathcal{B}}}
\begin{document}

\title{Stability of degenerate parabolic Cauchy problems}

\author{Teemu Lukkari and Mikko Parviainen}

\thanks{Part of this research was carried out at the Institut
  Mittag-Leffler (Djursholm). MP is supported by the Academy of Finland.}

\subjclass[2010]{35K55, 35K15, 35K65} 

\keywords{nonlinear parabolic equations, initial value problems,
  stability, Cauchy problems, Barenblatt solutions}

\begin{abstract}
  We prove that solutions to Cauchy problems related to the
  $p$-parabolic equations are stable with respect to the nonlinearity
  exponent $p$.  More specifically, solutions with a fixed initial
  trace converge in an $L^q$-space to a solution of the limit
  problem as $p>2$ varies.
\end{abstract}

\address[Teemu Lukkari]{Department of Mathematics and Statistics, P.O. Box 35 (MaD),
FI-40014 University of Jyv\"askyl\"a, Finland}
\email{teemu.j.lukkari@jyu.fi}

\address[Mikko Parviainen]{Department of Mathematics and Statistics,
  P.O. Box 35 (MaD), FI-40014 University of Jyv\"askyl\"a, Finland}
\email{mikko.j.parviainen@jyu.fi}

\date{\the\day.\the\month.\the\year}
\maketitle

\section{Introduction}

We study the stability of solutions to the degenerate parabolic
equation
\begin{equation}
  \label{eq:eq-intro}
  \partial_t u-\mathrm{div}(\abs{\nabla u}^{p-2}\nabla u)=0
\end{equation}
with respect to perturbations in the nonlinearity power $p$.  The main
issue we address is whether solutions converge in some sense to the
solution of the limit problem as the nonlinearity power $p$ varies.
In applications, parameters like $p$ are often only known
approximately, for instance from experiments.  Thus it is natural to
ask whether solutions are sensitive to small variations in such
parameters or not.

The stability of Dirichlet problems on bounded domains for
\eqref{eq:eq-intro} is detailed in \cite{kinnunenp10}.  See also
\cite{LqvistStab, LiMartio, LqvistRayleigh} for elliptic equations
similar to the $p$-Laplacian, \cite{lukkari} for the porous medium
and fast diffusion equations, and \cite{fujishimahkm} for parabolic quasiminimizers.  Here we focus on the Cauchy problem
\begin{displaymath}
  \begin{cases}
    \partial_t u-\mathrm{div}(\abs{\nabla u}^{p-2}\nabla u)=0,&
    \text{in }\R^n\times (0,\infty)\\
    u(x,0)=\nu
  \end{cases}
\end{displaymath}
with $p>2$. The initial trace $\nu$ can be a positive measure with
compact support and finite total mass, and the interpretation of the
initial condition is that the initial values is attained in the sense
of distributions. See Section \ref{sec:preliminaries} below for the
details.  For the existence theory of this initial value problem, we
refer to \cite{dibenedettoh89}.

An important example of solutions to the Cauchy problems is given by
the celebrated \emph{Barenblatt solutions}. These functions are the
nonlinear counterparts of the fundamental solution, and they are
defined for positive times $t>0$ by the formula
\begin{displaymath}
  \B_p(x,t)=t^{-n/\lambda}\left(C-\frac{p-2}{2}\lambda^{\frac{1}{1-p}}
    \left(\frac{\abs{x}}{t^{1/\lambda}}\right)^{\frac{p}{p-1}}\right)_+^{\frac{p-1}{p-2}},
\end{displaymath}
where $\lambda=n(p-2)+p$.  The constant $C$ is at our disposal. We
choose the value for which
\begin{displaymath}
  \int_{\R^n} \B_p(x,t)\,dx=1 \quad \text{for all}\quad t>0.
\end{displaymath}
With this normalization, one easily checks that
\begin{displaymath}
  \lim_{t\to 0} \B_p(\cdot,t)=\delta_0,
\end{displaymath}
in the sense of distributions, where $\delta_0$ is Dirac's delta at
the origin of $\R^n$. Further, the function $\B_p$ is the unique
solution to the Cauchy problem with this initial trace, see
\cite{kaminvaz88}.

Our main result states that solutions and their gradients converge to
a limit function in an $L^q$-space, locally in space and up to the
initial time, as $p$ varies. The limit function turns out to be a weak
solution to the equation involving the limit exponent, and it also has
the correct initial trace.  In particular, since the Barenblatt
solutions are the unique solutions to the respective Cauchy problems,
our stability theorem contains the fact that
\begin{equation}\label{eq:barenblatt-conv}
  \B_{\widetilde{p}}\to \B_{p} \quad\text{as} \quad \widetilde{p}\to p
  \quad\text{for} \quad p>2
\end{equation}
in the above sense.  The admissible exponents in the norms are sharp
in the sense that for larger exponents, the norms of the Barenblatt
solution and its gradient are no longer finite.  For general initial
measures, uniqueness for the Cauchy problem remains open, and we only
get a subsequence converging to \emph{some} solution of the limit
problem. If the initial trace is an $L^1$ function, then the solution
to the limit problem is unique, and the whole sequence converges.

The proof of our stability result is based on deriving suitable
estimates for solutions and their gradients. We use the truncation
device commonly employed in measure data problems for this purpose,
see e.g.~\cite{boccardodgo97}. This works well, since the mass of the
initial data has the same role in the estimates as the mass of a right
hand side source term. Convergence of the solutions is established by
employing the ideas of \cite{kinnunenp10} and a localization
argument. Here we encounter a difficulty: changing the growth exponent
$p$ changes the space to which solutions belong. To deal with this, we
need to use a local Reverse H\"older inequality for the gradient. Such
a result can be found in \cite{kinnunenl00}.

The paper is organized as follows. In Section~\ref{sec:preliminaries},
we present the necessary preparatory material, including the parabolic
Sobolev spaces to which weak solutions belong, and a rigorous
description of what is meant by the Cauchy problem in this context. In
Section \ref{sec:prop}, we prove a result concerning the propagation
properties of solutions. These properties are then employed in
Section~\ref{section:a priori estimates} in proving the estimates
necessary for our stability result. Finally, in
Section~\ref{sec:stability} we prove the main theorem about stability.

\section{Preliminaries} \label{sec:preliminaries}

We denote by 
$$
L^p(t_1,t_2;W^{1,p}(\Omega)),
$$
where $t_1<t_2$, the parabolic Sobolev space. This space consists of
measurable functions $u$ such that for almost every $t$, $t_1 < t <
t_2$, the function $x \mapsto u(x,t)$ belongs to $W^{1,p}(\Omega)$ and
the norm
$$
\left(\int_{t_1}^{t_2}\int_\Omega |u(x,t)|^{p} + |\nabla
u(x,t)|^{p}\,dx \,dt\right)^{1/p}
$$
is finite.  The definitions for
$L_{loc}^p(t_1,t_2;W_{loc}^{1,p}(\Omega))$ and
$L^p(t_1,t_2;W_{0}^{1,p}(\Omega))$ are analogous. The space
$C((t_1,t_2);L^q(\Omega))$, $q=1,2$, comprises of all the functions
$u$ such that for $t_1<s,t<t_2$, we have
\[
\int_{\Omega} |u(x,t)-u(x,s)|^q \, dx  \to 0
\]
as $s \to t$. The notation $U \Subset \Omega$ means that $U$ is a
bounded subset of $\Omega$ and the closure of $U$ belongs to
$\Omega$. We also denote $U_{\eps}=U\times(\eps,\infty)$.  We study
the Cauchy problem
\begin{equation} \label{eq:CP-trace}
\begin{cases}
\operatorname{div}\big( \abs{\nabla u}^{p-2}\nabla u\big)&=
\pdo{u}{t},\qquad  \textrm{in} \quad
\R^n \times (0,\infty),\\
 u(\cdot,0) & = \nu, 
\end{cases}
\end{equation}
where $\nu$ is a compactly supported positive Radon measure. For
simplicity, this measure is fixed throughout the paper. To be more
precise, the solution $u\ge 0$ satisfies the initial condition in the
sense of distributions, meaning that
\[
\lim_{t\to 0} \int_{\R^n}  u (x,t) \varphi(x) \, dx  = \int_{\R^n} \varphi(x) \, d\nu(x)
\]
for all $\vp \in C_0^\infty(\R^n)$. Further, $u$ is a weak solution
for positive times in the sense of the following definition.
\begin{definition} \label{def of weak sol}
A function
\[
u \in L^p_{\text{loc}}(0,\infty \, ; \, W^{1,p}_{\text{loc}}(\R^n))
\]
is a weak solution in $\R^n\times (0,\infty)$ if it satisfies
\begin{equation} \label{weak_solution}
\begin{split}
\int_{0}^{\infty}\int_{\Omega} \abs{\nabla u}^{p-2}\nabla u
\cdot \nabla\phi \, dx  \,dt -
\int_{0}^{\infty}\int_{\Omega} u
\frac{\partial\phi}{\partial t} \, dx \,dt
\: = & \: 0
\end{split}
\end{equation}
for every test function $\phi \in
C_0^\infty(\Rn\times(0,\infty))$.
\end{definition}

Definition of a solution can be also written in the form
\begin{equation} \label{weak_solutionb}
\begin{split}
\int_{\tau_1}^{\tau_2}&\int_{\Omega} \abs{\nabla u}^{p-2}\nabla u
\cdot \nabla\phi \, dx  \,dt -
\int_{\tau_1}^{\tau_2}\int_{\Omega} u
\frac{\partial\phi}{\partial t} \, dx \,dt
\\
&+ \int_{\Omega} u(x,\tau_2) \phi(x,\tau_2) \, dx  -
\int_{\Omega} u(x,\tau_1) \phi(x,\tau_1) \, dx
\: = \: 0
\end{split}
\end{equation}
for a test function $\phi \in C_0^\infty(\Rn\times(0,\infty))$ and for
almost every $0<\tau_1<\tau_2<\infty$.

Our assumptions imply that the Cauchy problem has a solution, see
\cite{dibenedettoh89} and \cite{kuusip09}. These solutions are
constructed by approximating the initial trace by regular functions,
proving uniform estimates, and then passing to the limit by
compactness arguments. 
\begin{remark}\label{rem:cpsol}
  By a solution to the Cauchy problem \eqref{eq:CP-trace} we mean any
  solution obtained as a limit of approximating the initial trace
  $\nu$ via standard mollifiers. The reason for this is that we want
  to choose the approximations so that their support does not expand
  much.
\end{remark}

The above definition does not include a time derivative of $u$ in any
sense. Nevertheless, we would like to employ test functions depending
on $u$, and thus the time derivative of $u$ inevitably appears. We
deal with this defect by using the convolution
\begin{equation}
\label{eq:naumann-convolution}
\begin{split}
u_\sigma(x,t)=\frac{1}{\sigma}\int_\eps^t e^{(s-t)/\sigma} u(x,s) \ud s,
\end{split}
\end{equation}
defined for $t\ge \eps>0$. This is expedient for our purpose; see for
example \cite{naumann84}, \cite{boccardodgo97}, and
\cite{kinnunenl06}. In the above, $\eps>0$ is a Lebesgue instant of
$u$. The reason for using a positive number $\eps$ instead of zero is
that solutions are not a priori assumed to be integrable up to the
initial time.

The convolution \eqref{eq:naumann-convolution} has the following
properties.
\begin{lemma}
\label{lem:naumann}
\begin{enumerate}[(i)]
\item \label{item:time-derivative-of-naumann}
If $u \in L^p(\Om_{\eps})$, then
\[
\begin{split}
\norm{u_\sigma}_{L^p(\Om_{\eps})}\leq \norm{u}_{L^p(\Om_{\eps})},
\end{split}
\]
\[
\begin{split}
\parts{u_\sigma}{t}=\frac{u-u_\sigma}{\sigma}\in L^p(\Om_{\eps}),
\end{split}
\]
and \[
\begin{split}
u_\sigma\to u \quad \trm{in} \quad L^p(\Om_{\eps})\quad \trm{as} \quad \sigma\to 0.
\end{split}
\]
\item If $\nabla u\in L^p(\Om_{\eps})$, then $\nabla u_\sigma=(\nabla u)_\sigma$ componentwise,
\[
\begin{split}
\norm{\nabla u_\sigma}_{L^p(\Om_{\eps})}\leq \norm{\nabla u}_{L^p(\Om_{\eps})},
\end{split}
\]
and
\[
\begin{split}
\nabla u_\sigma\to \nabla u \quad \trm{in} \quad L^p(\Om_{\eps})\quad \trm{as} \quad \sigma\to 0.
\end{split}
\]
\item
Furthermore, if $u^k\to u$ in $L^p(\Om_{\eps})$, then also
\[
\begin{split}
u^k_\sigma\to u_\sigma\quad \trm{and}\quad \parts{u^k_\sigma}{t}\to \parts{u_\sigma}{t}
\end{split}
\]
in $L^p(\Om_{\eps})$.
\item If $\nabla u^k\to \nabla u$ in $L^p(\Om_{\eps})$, then $\nabla u^k_\sigma\to \nabla u_\sigma$ in $L^p(\Om_{\eps})$.
\item Analogous results hold for the weak convergence in $L^p(\Om_{\eps})$.
\item Finally, if $\vp\in C(\ol \Om_{\eps})$, then
\[
\begin{split}
\vp_\sigma(x,t) +e^{- \frac{t-\eps}{\sigma}} \vp(x,\eps) \to \vp(x,t)
\end{split}
\]
\emph{uniformly} in $\Om_{\eps}$ as $\sigma\to 0$.
\end{enumerate}
\end{lemma}

To derive the necessary estimates, we need the equation satisfied by
the mollified solution $u_\sigma$. By using \eqref{weak_solutionb}, we
may write
\[
\begin{split}
\int_{s}^{T-\eps}&\int_{\Rn} \abs{\nabla u(x,t+\eps-s)}^{p-2}\nabla u(x,t+\eps-s)
\cdot \nabla\phi(x,t+\eps) \, dx  \,dt \\
&\hspace{1 em}-
\int_{s}^{T-\eps}\int_{\Rn} u(x,t+\eps-s)
\frac{\partial\phi(x,t+\eps)}{\partial t} \, dx \,dt
\\
&\hspace{1 em}+ \int_{\Rn} u(x,T-s)\phi(x,T) \, dx\\
& =
\int_{\Rn} u(x,\eps) \phi(x,s+\eps) \, dx
\end{split}
\]
when $0\leq s\leq T-\eps$. Notice that $t+\eps-s\in [\eps,T-s] $. Then we
multiply the above inequality by $e^{-s/\sigma}/\sigma$, integrate
over $[0,T-\eps]$ with respect to $s$, change the order of integration
using $\int_0^{T-\eps}\int_{s}^{T-\eps} \ldots \ud t \ud
s=\int_0^{T-\eps} \int_0^t \ldots \ud s \ud t$, and finally perform a
change of variables $s_{\text{new}}=s+\eps, t_{\text{new}}=t+\eps$, and for simplicity denote the new variables also with $s$ and $t$ . We end up with
\begin{equation}
\label{eq:reg}
\begin{split}
\int_{\eps}^{T}&\int_{\Rn} (\abs{\nabla u(x,t)}^{p-2}\nabla u(x,t))_\sigma
\cdot \nabla\phi(x,t) \, dx  \,dt \\
&-
\int_{\eps}^{T}\int_{\Rn} u(x,t)_{\sigma}
\frac{\partial\phi(x,t)}{\partial t} \, dx \,dt
+ \int_{\Rn} u(x,T)_{\sigma} \phi(x,T) \, dx\\
& =
\int_{\Rn} u(x,\eps) \int_\eps^{T}\phi(x,s) e^{-(s-\eps)/\sigma}/\sigma\,\ud s\ dx.
\end{split}
\end{equation}
It is often convenient to write the last two terms on the right as 
\begin{displaymath}
  \begin{aligned}
    -\int_{\eps}^{T}\int_{\Rn} u(x,t)_{\sigma}
    \frac{\partial\phi(x,t)}{\partial t} \, dx \,dt&
    + \int_{\Rn} u(x,T)_{\sigma} \phi(x,T) \, dx\\&=\int_{\eps}^T\int_{\R^n}
    \frac{\partial u _\sigma (x,t) }{\partial t}\phi(x,t)\,dx\,dt,
  \end{aligned}
\end{displaymath}
where we integrated by parts in time and used the fact that
$u_\sigma(x,\eps)$ vanishes by the properties of the mollifier.

\section{Propagation properties}
\label{sec:prop}

In this section, we prove a finite propagation result for the Cauchy
problem. More specifically, if the initial trace has compact support,
then this property is preserved in the time evolution, meaning that
$x\mapsto u(x,t)$ also has compact support for positive times $t$.
This is in principle well known, but it seems hard to find a
convenient reference covering the case when the initial trace is a
measure. For stronger initial conditions, see for example Theorem 3.1
in \cite{kuusip09}.

 Here we use a comparison with an explicit solution and  adapt the arguments for the porous medium
equation in \cite{VB}, see in particular Lemma~14.5 and
Proposition~14.24.
We start by extracting a family of barrier functions from the
Barenblatt solution
\begin{displaymath}
  \B_p(x,t)=t^{-n/\lambda}\left(C-\frac{p-2}{2}\lambda^{1/(1-p)}
    \left(\frac{\abs{x}}{t^{1/\lambda}}\right)^{p/(p-1)}\right)_+^{(p-1)/(p-2)},
\end{displaymath}
where $\lambda=n(p-2)+p$. The Barenblatt solution is of the form
\begin{displaymath}
  \B_p(x,t)=F(t,\abs{x}^{p/(p-1)})
\end{displaymath}
for a suitable function $F$. From this, we see that the function
\begin{displaymath}
  v(x,t)=F(T-t,-\abs{x-x_1}^{p/(p-1)})
\end{displaymath}
is also a weak solution, since both the terms in the $p$-parabolic equation change
sign. With the choice $C=0$ and some simplifications, we get a weak solution
\begin{equation}\label{eq:barrier}
  v(x,t)=c(n,p)(T-t)^{-\frac{1}{p-2}}\abs{x-x_1}^{\frac{p}{p-2}}.
\end{equation}

The first result concerns finite propagation for bounded initial data.
\begin{theorem}\label{prop:bounded-propagation}
  Let $u$ be the solution to the Cauchy problem 
  \begin{displaymath}
    \begin{cases}
      u_t-\divt(\abs{\nabla u}^{p-2}\nabla u)=0,\\
      u(x,0)=u_0(x)
    \end{cases}
  \end{displaymath}
  where 
  \begin{displaymath}
    u_0\geq 0 \quad\text{and} \quad\norm{u_0}_{\infty}=H<\infty.
  \end{displaymath}
  Assume that 
  there is a point $x_0\in \R^n$ and a radius $R>0$ such that
  \begin{displaymath}
    u_0(x)=0 \quad\text{for all}\quad x\in B(x_0,R).
  \end{displaymath}
  Then there exists a nonnegative function $R(t)$ such that 
   \begin{displaymath}
    R(t)\geq R-c(n,p)H^{\frac{p-2}{p}}t^{1/p},
  \end{displaymath}
  and
  \begin{displaymath}
    u(x,t)=0 \quad\text{for all}\quad x\in B(x_0,R(t)).
  \end{displaymath}
 \end{theorem}
 \begin{remark}
   It can be shown that the function $R(t)$ is nonincreasing, but we
   do not need this fact here.  In
   Theorems~\ref{prop:bounded-propagation}
   and~\ref{thm:general-propagation}, it can happen that $R(t)=0$
   after some time instant.
 \end{remark}
\begin{proof}
  By the comparison principle, we have $0\leq u\leq H$. Pick a point
  $x_1\in B(x_0,R)$ and denote 
  \begin{displaymath}
    d(x_1)=\dist(x_1,\partial B(x_0,R))=R-\abs{x_1-x_0}.
  \end{displaymath}

  We derive an estimate for the time $T$ up to which $u(x_1,T)=0$ by
  comparing $u$ with the barrier function $v$ given by
  \eqref{eq:barrier} in $B(x_0,R)\times(0,T)$. At the initial time
  zero, $u\leq v$ since $u$ vanishes. On the lateral boundary
  $\partial B(x_0,R)\times(0,T)$ we have $u\leq v$ if
  \begin{equation}\label{eq:comparison-cond}
    H\leq c T^{-\frac{1}{p-2}}d(x_1)^{\frac{p}{p-2}}.
  \end{equation}
  Choose now $T$ small enough so that equality holds in
  \eqref{eq:comparison-cond}. Then
  \begin{displaymath}
    d(x_1)=cH^{\frac{p-2}{p}}T^{1/p},
  \end{displaymath}
  from which we get
  \begin{displaymath}
    \abs{x_1-x_0}=R-cH^{\frac{p-2}{p}}T^{1/p}.
  \end{displaymath}

  The choice of $T$ made above is admissible in \eqref{eq:comparison-cond} 
  for all $\tilde{x}$ such that $d(\tilde{x})\geq d(x_1)$.  Thus
  $u(\tilde x,t)=0$ for all $0\leq t\leq T$ and $\tilde x\in
  B(x_0,R(T))$ where
  \begin{displaymath}
    R(T)\geq R-cH^{\frac{p-2}{p}}T^{1/p},
  \end{displaymath}
  as desired.
\end{proof}

The next step is to generalize the above result to the case of initial
data with finite mass. For this purpose, we need the so-called
$L^1-L^\infty$ regularizing effect, i.e. the estimate
\begin{displaymath}
  \norm{u(\cdot,t)}_{\infty}\leq c(n,p)\nu(\R^n)^{\sigma}t^{-\alpha}.
\end{displaymath}
Here $\lambda = n(p-2)+p$, $\alpha=n/\lambda$, and $\sigma=p/\lambda$,
see Theorem 2.1 in \cite{dibenedetto93} on page 318. The proof is an
iteration of the estimate in Theorem \ref{prop:bounded-propagation}
with the help of the regularizing effect.

\begin{theorem}\label{thm:general-propagation}
  Assume that $\nu$ is a positive measure with $\nu(\R^n)<\infty$, and
  that there is a point $x_0\in \R^n$, a radius $R>0$, and a number
  $\varepsilon>0$ such that
  \begin{displaymath}
    \nu(B(x_0,R+\varepsilon))=0.
  \end{displaymath}
  Then there is a solution $u$ to the Cauchy problem
  \begin{displaymath}
    \begin{cases}
      u_t-\divt(\abs{\nabla u}^{p-2}\nabla u)=0,\\
      u(x,0)=\nu
    \end{cases}
  \end{displaymath}
  Then there exists a nonnegative function $R(t)$ such that
  \begin{displaymath}
    R(t)\geq R-c(n,p,\nu) t^{\gamma},
  \end{displaymath}
  and
  \begin{displaymath}
    u(x,t)=0 \quad\text{for all}\quad x\in B(x_0,R(t))
  \end{displaymath}
  where
  \begin{displaymath}
    \gamma=\frac{1}{p}\left(1-\frac{n(p-2)}{n(p-2)+p}\right)\quad
    \text{and}\quad 
    c(n,p,\nu)=c(n,p)\nu(\R^n)^{\sigma(p-2)/p}.
  \end{displaymath}
\end{theorem}
\begin{proof}
  We derive the estimate for the case where the initial trace is a
  bounded function $u_0$ such that
  \begin{displaymath}
    u_0(x)=0 \quad\text{for all}\quad x\in B(x_0,R).
  \end{displaymath}
  The general case then follows by mollifying the initial trace $\nu$
  so that the above assumption holds for the approximations, and
  passing to the limit. Indeed, the estimate depends on the initial
  trace only via its total mass, so the result holds also for the
  limit function.

  Fix a small $t>0$ and let $t_k=2^{-k}t$ for $k=0,1,2,\ldots$. Then
  $t_{k-1}=2t_k$ and $t_{k-1}-t_k=t_k$. Denote $
  H_k=\norm{u(\cdot,t_k)}_\infty$; then
  \begin{displaymath}
    H_k^{(p-2)/p}\leq c\nu(\R^n)^{\sigma(p-2)/p} t^{-\alpha(p-2)/p}2^{k\alpha(p-2)/p}
  \end{displaymath}
  by the regularizing effect. Further, we have
  \begin{displaymath}
    t_k^{1/p}=c2^{-k/p}t^{1/p}.
  \end{displaymath}
  
  We combine the estimate of Theorem \ref{prop:bounded-propagation}
  over the time interval $(t_{k+1},t_k)$ and the last two display
  formulas to get
  \begin{displaymath}
    R(t_k)\geq R(t_{k+1})-c\nu(\R^n)^{\sigma(p-2)/p}t^{\gamma}2^{-k\gamma}
  \end{displaymath}
  where we used the fact that
  \begin{displaymath}
    \frac{1}{p}-\alpha\frac{p-2}{p}=
    \frac{1}{p}\left(1-\frac{n(p-2)}{n(p-2)+p}\right)=\gamma>0.
  \end{displaymath}
  Iteration of the above estimate leads to
  \begin{displaymath}
    R(t)=R(t_0)\geq 
    R(0)-c(n,p,\nu)t^{\gamma}
    \sum_{k=1}^{\infty} 2^{-k\gamma}.
  \end{displaymath}
  The series is convergent since $\gamma$ is positive, and the proof
  is complete.
\end{proof}

\begin{remark}\label{rem:cptsupport}
  We will use Theorem \ref{thm:general-propagation} in the following
  way: given a time $T$ and an initial trace $\nu$ with compact
  support, we can find a bounded open set $U$ such that the support of
  $u(\cdot,t)$ is contained in $U$ for all $0< t\leq T$. To see this,
  note that we may choose the point $x_0$ in Theorem
  \ref{thm:general-propagation} to be arbitrarily far away from the
  support of $\nu$.
\end{remark}

\section{A priori estimates} \label{section:a priori estimates}

In this section, we derive some estimates we will employ for our
stability result. More specifically, we estimate certain $L^q$-norms
of solutions and their gradients locally in space, up to the initial
time.  We begin with a simple lemma for passing up to the initial time
in our estimates.
\begin{lemma}\label{lem:initial-estimate}
  Let $u$ be a solution to the Cauchy problem with initial trace
  $\nu$, and let $U\Subset \R^n$ be an open set. Then
  \begin{displaymath}
    \limsup_{\eps\to 0}\int_{U}u(x,\eps)\, dx\leq \nu(\overline{U}).
  \end{displaymath}
\end{lemma}
\begin{proof}
  Since $u$ is a local weak solution for positive times, it is also
  continuous in time with values in $L^2_{loc}(\R^n)$ on
  $(0,\infty)$. Thus $x\mapsto u(x,\eps)$ is a locally integrable
  function for \emph{every} $\eps>0$, and we may identify it with the
  measure defined by
  \begin{displaymath}
    \nu_\eps(E)=\int_E u(x,\eps)\,dx.
  \end{displaymath}
  Observe that we need the continuity in time only for strictly
  positive times.  Now the fact that $\nu$ is the initial trace of $u$
  means exactly that $\nu_\eps\to \nu$ in the sense of weak
  convergence of measures. Thus
  \begin{displaymath}
    \limsup_{\eps \to 0}\int_U u(x,\eps)\,dx\leq \limsup_{\eps \to
      0}\int_{\overline{U}}u(x,\eps)\,dx \leq \nu(\overline{U})
  \end{displaymath}
  by the properties of weak convergence of measures.
\end{proof}

\begin{lemma}\label{lem:gradient-est}
  Let $u$ be a solution to \eqref{eq:CP-trace}. Then there is a
  constant $c=c(n,p)$ such that
  \begin{align*}
    \sup_{0<t<T}\int_U \min(u,j)^2\ud x
    &+\int_{U_T}\abs{\nabla \min(u,j)}^p\ud x\ud t\le cj \nu(\Rn).
  \end{align*}
\end{lemma}
\begin{remark}\label{rem:gradient-est}
  We can replace the left hand side by 
  \begin{displaymath}
    \sup_{\eps<t<T}\int_U \min(u,j)^2\ud x
    +\int_{U_{\eps,T}}\abs{\nabla \min(u,j)}^p\ud x\ud t
  \end{displaymath}
  by a trivial estimate.
\end{remark}
\begin{proof}
  We test the regularized equation with $\varphi= u_j=\min(u,j)$. This
  is an admissible test function, since $u$ has a 'finite speed of
  propagation' i.e. for a given time $T$, there is $U$ such that
  $u(\cdot, t)$ is supported on $U$, see Remark
  \ref{rem:cptsupport}. We get
  \begin{align*}
    \int_{U_{\eps,T}}\frac{\partial u_\sigma}{\partial t}u_j
    &+(\abs{\nabla u}^{p-2}\nabla u)_\sigma \cdot\nabla u_j \ud x\ud t\\
    = & \int_U u(x,\eps)\left(\frac{1}{\sigma}\int_\eps^Te^{-(s-\eps)/\sigma}
      u_j(x,s)\ud s\right)\ud x.
  \end{align*}
  
  We need to eliminate the time derivative in the first term on the left. 
  To accomplish this, we note that
  \begin{align*}
    \frac{\partial u_\sigma}{\partial t}u_j=&
     \frac{\partial u_\sigma}{\partial t}\min(u_\sigma,j)+
    \frac{\partial u_\sigma}{\partial t}(\min(u,j)-\min(u_\sigma,j))\\
    \geq & \frac{\partial u_\sigma}{\partial t}\min(u_\sigma,j),
  \end{align*}
  so that this term is estimated from below by
  \[
\begin{split}
\int_{U_{\eps,T}}\frac{\partial u_\sigma}{\partial t}u_j \ud x \ud t&=     \int_{U_{\eps,T}}\frac{\partial }{\partial t}\int_0^{u_{\sigma}(x,t)}\min(s,j) \ud s\, \ud t\ud x\\
     &=
    \int_{U} F(u_\sigma(x,T))\ud x-\int_{U} F(u_\sigma(x,\eps))\ud x
\end{split}
\]
  where 
  \begin{displaymath}
    F(t)=\int_0^t\min(s,j)\ud s=
    \begin{cases}
      \frac{1}{2}t^2, & t\leq j, \\
      \frac{1}{2}j^2+jt-j^2\geq \frac{1}{2}j^2, &t>j.
    \end{cases}
  \end{displaymath}
  In particular, $F(t)\geq \frac{1}{2}\min(t,j)^2$.
  Since $u_\sigma(x,\eps)=0$, we get 
  \begin{displaymath}
    \int_{U_{\eps,T}}\frac{\partial u_\sigma}{\partial t}u_j\ud x\ud t\geq 
    \int_{U}\frac{1}{2}\min(u_\sigma,j)^2(x,T)\ud x.
  \end{displaymath}
  
  The time derivative has been eliminated, so we may pass to the limit
  $\sigma \to 0$. We take this limit and get
  \begin{align*}
    \frac{1}{2}\int_U&u_j^2(x,T)\ud x
    +\int_{U_{\eps,T}}\abs{\nabla u_j}^p\ud x\ud t \leq \int_U u(x,\eps) u_j(x,\eps)\ud x.
  \end{align*}
  We then use the fact that $u_j\leq j$
  \begin{align*}
    \int_U & u_j^2(x,T)\ud x
    +\int_{U_{\eps,T}}\abs{\nabla u_j}^p \ud x\ud t \leq c j \int_U u(x,\eps)\ud x.
  \end{align*}
  Then pass to the limit $\eps\to 0$; the right hand side is bounded
  by $cj\nu(\overline{U})$ by Lemma~\ref{lem:initial-estimate}, so we
  get the desired quantity to the right hand side by a trivial
  estimate.
  
  The proof is then completed by replacing $T$ in the above by $\tau$
  chosen so that
  \begin{displaymath}
    \int_U u_j^2(x,\tau)\ud x\geq 
    \frac{1}{2}\esssup_{0<t<T}\int_U u_j^2(x,t)\ud x.
  \end{displaymath}
  This leads to an estimate for the supremum in terms of the right
  hand side in the claim. 
\end{proof}

We need the following well known Sobolev type inequality.
\begin{lemma}\label{lem:sobo}
  Let $u\in L^p(\eps,T;W^{1,p}_0(U))$. Then
  \begin{displaymath}
    \int_{U_{\eps,T}}\abs{u}^{\kappa p}\ud x\ud t\leq c\int_{U_{\eps,T}}\abs{\nabla u}^p\ud x\ud t
    \left(\esssup_{\eps<t<T} \int_U u^2(x,t)\ud x\right)^{p/n},
  \end{displaymath}
  where
  \begin{displaymath}
    \kappa =\frac{n+2}{n}.
  \end{displaymath}
\end{lemma}

Next we show that solutions to the Cauchy problem, as well as their
gradients, are integrable to certain powers. The proof is based on
estimating the decay of certain level sets by applying the previous
two lemmas.  This is optimal as can be seen from the Barenblatt
solution. The essential point is that the constant on the right can be
chosen to be independent of $u$.
\begin{theorem}\label{thm:gradient-est}
  Let $u$ be a solution to \eqref{eq:CP-trace}. Then
  \[
\begin{split}
  \int_0^T\int_{\Rn} &u^q \ud x \ud t\leq  C(\nu,p,q)<\infty  \quad \text{whenever} \quad q<p-1+\frac{p}{n},\text{ and}\\
  \int_0^T\int_{\Rn}&\abs{\nabla u}^q \ud x \ud t\leq  C(\nu,p,q)<\infty  \quad \text{whenever} \quad q<p-1+\frac{1}{n+1}.
\end{split}
\]

Above the constant remains bounded when $p\to p_0>2$. 
\end{theorem}
\begin{proof}
  Since $T$ and $\nu$ are given, we can choose open $U\Subset \Rn$
  such that $u(\cdot,t)$ is compactly supported in $U$ for each
  $t\in[0,T]$, see Remark \ref{rem:cptsupport}. Let us define the sets
  \begin{displaymath}
    E_j^\eps=\{(x,t)\in U_{\eps,T}: j\leq u(x,t)<2j\}.
  \end{displaymath}
  The reason for introducing a positive $\eps>0$ here is that we only
  know the integrability of the solution $u$ and its gradient for
  strictly positive times. The aim is to derive estimates independent
  of $\eps$ on the time interval $(\eps,T)$, and then pass to the
  limit $\eps\to 0$. This is possible since the right hand side in the
  estimate of Lemma~\ref{lem:gradient-est} is independent of $\eps$.

  Recall the notation
  $u_{2j}=\min(u,2j)$. We have
  \begin{align*}
    j^{\kappa p}\abs{E_j^\eps}\leq & \int_{E_j^\eps}u^{\kappa p}_{2j}\ud x\ud t\\
    \leq & \int_{U_{\eps,T}}u^{\kappa p}_{2j}\ud x\ud t\\
    \leq & c\int_{U_{\eps,T}}\abs{\nabla u_{2j}}^p\ud x\ud
    t\left(\esssup_{\eps<t<T}\int_U u_{2j}^2 \ud x\right)^{p/n}\\
    \leq & c\nu(\R^n)j^{1+p/n}
  \end{align*}
  by applications of the parabolic Sobolev inequality and Lemma
  \ref{lem:gradient-est}; see also Remark \ref{rem:gradient-est}. It
  follows that
  \begin{equation}\label{eq:decay}
    \abs{E_j^\eps}\leq cj^{1-p-p/n}
  \end{equation}
  with a constant independent of $\eps$.

  With \eqref{eq:decay}, the estimates follow by essentially the same
  computations as in the proof of Lemma 3.14 in \cite{kinnunenl05}. By
  \eqref{eq:decay}, we get
  \[
  \begin{split}
    \int_\eps^T\int_{\Rn} u^q \ud x \ud t&\le \int_\eps^T\int_{\{u<1\}} u^{q-1} u  \ud x \ud t+\sum_{j=1}^\infty \int_{E_{2^{j-1}}^\eps} 
    u^q \ud x \ud t\\
    &\le \nu(\Rn)T+\sum_{j=1}^\infty 2^{jq}\abs{E_{2^{j-1}}^\eps}\\
    & \le \nu(\Rn)T+2^{p+p/n-1}\sum_{j=1}^\infty  2^{j(q+1-p-p/n)}<\infty.
  \end{split}
\]
The sum on the last line is finite by the choice of $q$. We let
$\eps\to 0$ to conclude the desired estimate for $u$.

Next we estimate the gradient. We use Hölder's inequality,
\eqref{eq:decay} and Lemma \ref{lem:gradient-est} to obtain
  \[
\begin{split}
\int_\eps^T &\int_U \abs{\nabla u}^q \ud x\ud t\\
&=\int_{\{u<1\}} \abs{\nabla u_1}^q\ud x \ud t  +\sum_{j=1}^\infty\int_{E_{2^{j-1}}^\eps} \abs{\nabla u}^q \ud x \ud t\\
&\le \nu(\Rn)^{q/p}(T\abs{U})^{1-q/p}+\sum_{j=1}^\infty\Big(\int_{E_{2^{j-1}}^\eps} \abs{\nabla u}^p \ud x \ud t\Big)^{q/p}\abs{E_{2^{j-1}}^\eps}^{1-q/p}\\
&\le \nu(\Rn)^{q/p}(T\abs{U})^{1-q/p}\\
&\hspace{1 em}+\sum_{j=1}^\infty 2^{(j-1)(1-q/p)(1-p-p/n)}\Big(\int_{E_{2^{j-1}}^\eps} \abs{\nabla u_{2^{j-1}}}^p \ud x \ud t\Big)^{q/p}.
\end{split}
\]
We apply Lemma \ref{lem:gradient-est} once more to estimate the last
integral by $C 2^{j-1}$. This gives the bound 
\[
\begin{split}
C \sum_{j=1}^\infty 2^{(j-1)(1-q/p)(1-p-p/n)}2^{(j-1)q/p}=C\sum_{j=1}^\infty 2^{(j-1)(1-p-p/n+q+q/n)}
\end{split}
\] 
for the second term. Here the sum converges since
$1-p-p/n+q+q/n<0$. Letting $\eps\to 0$ completes the proof.
\end{proof}

\section{Stability}

\label{sec:stability}

Our stability result now follows from the estimates of the previous
section by arguments similar to those in \cite{kinnunenp10}.

We work with a sequence of exponents $(p_i)$ such that 
\begin{displaymath}
  \lim_{i\to \infty}p_i=p>2.
\end{displaymath}
By the convergence, we are free to assume that $p_i>2$, and that all
the exponents belong to a compact subinterval $[p^-,p^+] $of
$(2,\infty)$. 

\begin{theorem}
  \label{thm:weak-conv}
  Let $(p_i)$ be a sequence such that $p_i\to p>2$, $(u_i)$ the
  corresponding solutions to \eqref{eq:CP-trace} (see Remark
  \ref{rem:cpsol}), $T>0$ and $U\Subset \Rn$.Then there exists a
  subsequence still denoted by $(u_i)$ and a function $u\in
  L^{q}(0,T;W^{1,q}(U))$, $q<p-1+1/(n+1)$, such that
  \[
  u_i\to u \quad\trm{in}\quad
  L^{q}({U_T})
  \]
  and
  \[
  \nabla u_i\to \nabla u \quad\trm{weakly
    in}\quad L^{q}({U_T}),
  \] as $i\to\infty$.
\end{theorem}
\begin{proof}
  Since $q>1$, weak convergence follows from the reflexivity of $L^q$
  and the uniform $L^q$-bounds of Theorem \ref{thm:gradient-est}. For the norm convergence of the solutions,
  we estimate the time derivative of $u$ and appeal to the parabolic version of Rellich's theorem by Simon 
  \cite{simon87}, see also
 page 106 of Showalter's monograph \cite{showalter97}. Denote $1/q'+1/q=1$.
  For all $\varphi\in C_0^\infty(U_T)$, we have
\[
\begin{split}
    \abs{\inprod{\partial_t u_i,\varphi}}&=
    \abs{\int_{U_T}\abs{\nabla u_i}^{p-2}\nabla u_i\cdot \nabla \varphi\ud x \ud t}
    \\
    &\leq C \Big(\int_{U_T}\abs{\nabla u_i}^{(p-1)q'} \ud x \ud t\Big)^{1/q'}\norm{\nabla \varphi}_{L^{q}(U_T)}
\end{split}
\]
  for any $\phi\in C^{\infty}_0(U_T)\cap
L^q(0,T;W_0^{1,q}(U))$. Since $(p-1)q'<q$, it follows that $\Big(\int_{U_T}\abs{\nabla u_i}^{(p-1)q'} \ud x \ud t\Big)^{1/q'}<C$. Thus, by the density of smooth functions in $L^q(0,T;W_0^{1,q}(U))$,
the time derivative of $u_i$ is bounded in the dual space of
$L^q(t_1,t_2;W^{1,q}_0(\Omega'))$.
  Now the theorem mentioned above implies the claim.
\end{proof}

The previous theorem is not sufficient for our stability result, since
the weak convergence of the gradients is not enough to identify the
weak limit of the nonlinear quantity $\abs{\nabla u_i}^{p_i-2}\nabla
u_i$. This is rectified in the following theorem by proving the
pointwise convergence of the gradients.

\begin{theorem}
\label{thm-main}
Let $(p_i)$, $(u_i)$, $q$, $U$  be as in Theorem~\ref{thm:weak-conv}.
There exists a
subsequence $(u_i)$ and a function $u\in
L^{q}(0,T;W^{1,q}(U))$ such that
\[
\begin{split}
u_i \to u \quad\trm{in}\quad
L^{q}(0,T;W^{1,q}(U)),
\end{split}
\] as $i\to\infty$.
\end{theorem}

\begin{proof}
  By Remark \ref{rem:cptsupport}, we can again assume that $\spt
  u(\cdot,t)\subset U$.  We can focus our attention to the convergence
  of the gradients in $L^{q}(U_T)$.  To establish this, let $u_j$ and
  $u_k$ be two solutions in the sequence. Since both $u_j$ and $u_k$
  satisfy the mollified equation \eqref{weak_solutionb}, by
  subtracting, we obtain
\begin{equation} \label{summa}
\begin{split}
\int_{{U_{\eps,T}}}& \parts{}{t}(u_j-u_k)_{\sigma} \phi \ud x \ud t\\
& \hspace{1 em}+\int_{U_{\eps,T}} (\abs{\nabla u_j}^{p_j-2}\nabla
u_j-\abs{\nabla u_k}^{p_k-2}\nabla u_k)_{\sigma} \cdot \nabla \phi
\ud x \ud t\\
&=\int_{U} (u_j(x,\eps)-u_k(x,\eps))\frac1\sigma \int_{\eps}^T e^{-(s-\eps)/\sigma}\phi(x,s) \ud s \ud x. 
\end{split}
\end{equation}
We use  the test function
\[
\phi(x,t)=u_j(x,t)-u_k(x,t).
\]
 Also observe that $p_j,p_k$ are close enough to $p$, both the functions actually belong to a higher parabolic Sobolev space with power $p+\delta$ with some $\delta>0$ by the higher integrability result in \cite{kinnunenl00}.


We estimate the first term on the left hand side of \eqref{summa}. A
substitution of the test function and the properties of the
convolution imply
\[
\begin{split}
\int_{{U_{\eps,T}}}& \parts{}{t}(u_j-u_k)_{\sigma} \phi \ud x \ud t\\
&=
\int_{{U_{\eps,T}}} \parts{}{t}(u_j-u_k)_{\sigma} (u_j-u_k) \ud x \ud t\\
&= 
\int_{{U_{\eps,T}}} \Big((u_j-u_k)-(u_j-u_k)_\sigma\Big)^2 \ud x \ud t \\  
&\hspace{1 em}+\int_{{U_{\eps,T}}}\parts{}{t}(u_j-u_k)_{\sigma} (u_j-u_k)_\sigma\ud x \ud t\\
&\ge   
\frac12 \int_{{U}}(u_j-u_k)^2_{\sigma}(x,T) \ud x-\frac12\int_{{U}}(u_j-u_k)^2_{\sigma}(x,\eps) \ud x.
\end{split}
\]
This estimate is free of the time derivatives of the functions
$u_j$ and $u_k$, which is essential for us in the passage to the limit with
$\sigma$. Now, letting $\sigma\to 0$, we conclude that
\[
\begin{split}
\int_{{U_{\eps,T}}}& (\abs{\nabla u_j}^{p_j-2}\nabla
u_j-\abs{\nabla u_k}^{p_k-2}\nabla u_k) \cdot
(\nabla u_j-\nabla u_k) \ud x \ud t\\
\leq & -\int_{U} (u_j-u_k)^2(x,T)  \ud x
 +\int_{U} (u_j-u_k)^2(x,\eps)  \ud x.
 \end{split}
\]
Observe that the first term on the right hand side is
nonpositive. Thus
\begin{equation}
\label{eq-smaller-than-zero}
\begin{split}
 \int_{{U_{\eps,T}}}& (\abs{\nabla
u_j}^{p_j-2}\nabla u_j-\abs{\nabla u_k}^{p_k-2}\nabla u_k) \cdot
(\nabla u_j-\nabla u_k) \eta \ud x \ud t\\
&\le \int_{U} (u_j-u_k)^2(x,\eps) \eta(x) \ud x.
\end{split}
\end{equation}
We divide the left hand side in three parts as
\beq
\label{eq-three-parts}
\begin{split}
\int_{{U_{\eps,T}}}& (\abs{\nabla u_j}^{p_j-2}\nabla
u_j-\abs{\nabla u_k}^{p_k-2}\nabla u_k) \cdot
(\nabla u_j-\nabla u_k)\ud x \ud t \\
&= \int_{{U_{\eps,T}}} (\abs{\nabla u_j}^{p-2}\nabla
u_j-\abs{\nabla u_k}^{p-2}\nabla u_k) \cdot
(\nabla u_j-\nabla u_k)\ud x \ud t \\
&\hspace{1 em}+ \int_{{U_{\eps,T}}} (\abs{\nabla
u_j}^{p_j-2}-\abs{\nabla
  u_j}^{p-2})\nabla
u_j \cdot
(\nabla u_j-\nabla u_k) \ud x \ud t\\
&\hspace{1 em}+\int_{{U_{\eps,T}}} (\abs{\nabla
u_k}^{p-2}-\abs{\nabla u_k}^{p_k-2})\nabla u_k \cdot
(\nabla u_j-\nabla u_k) \ud x \ud t\\
&=I_1+I_2+I_3.
\end{split}
\eeq
First, we concentrate on $I_2$ and $I_3$.
A straightforward calculation shows that
\[
\begin{split}
|\abs{\zeta}^a-\abs{\zeta}^b|
&=\abs{\exp(a\log\abs{\zeta})
  -\exp(b\log(\abs{\zeta})}\\
&\leq \max_{s\in [a,b]} \abs{\parts{\,\exp(s
  \log\abs{\zeta})}{s}} \abs{a-b}\\
  & \leq \abs{\log\abs{\zeta}} ( \abs{\zeta}^a+
\abs{\zeta}^b)\abs{a-b},
\end{split}
\]
where $\zeta\in\Rn$ and $a,\,b\ge0$.
If $\abs{\zeta}\geq 1$, then
\[
\abs{\log\abs{\zeta}} ( \abs{\zeta}^{a}+
\abs{\zeta}^b)
\leq \frac1\gamma\abs{\zeta}^{\max(a,b)+\gamma},
\]
and if $\abs{\zeta}\leq 1$, then
\[
 \abs{\log\abs{\zeta}} ( \abs{\zeta}^{a}+
\abs{\zeta}^b) \leq \frac 1e\left(\frac1a+\frac1b\right).
\]
This leads to
\begin{equation}
\label{eq:any_z}
||\zeta|^a-|\zeta|^b|
\leq\left(\frac1\gamma\abs{\zeta}^{\max(a,b)+\gamma}+\frac 1e\left(\frac1a+\frac1b\right)
\right)\abs{a-b}
\end{equation}
for every $\zeta\in\Rn$ and $a,\,\ge0$.

Next we apply \eqref{eq:any_z} with $\zeta=\nabla u_j$, $a=p_j-2$,
and $b=p-2$. This implies
\[
\begin{split}
|I_2|&\leq c \abs{p_j-p} \int_{{U_{\eps,T}}} (1+\abs{\nabla
  u_j}^{\max(p_j-2,p-2)+\gamma})\abs{ \nabla u_j}
\abs{\nabla u_j-\nabla u_k} \ud x \ud t.
\end{split}
\]
The integral on the right hand side is uniformly bounded for small enough $\gamma>0$ by Theorem \ref{thm:gradient-est}. Consequently, $I_2\to 0$, as $j,k\to \infty$. A similar reasoning
implies that $I_3$ tends to zero as $j,k\to\infty$. From the
elementary inequality
 \[
2^{2-p}\abs{a-b}^p\leq (\abs{a}^{p-2} a-\abs{b}^{p-2} b)\cdot
(a-b),
\]
 \eqref{eq-smaller-than-zero} as well as
\eqref{eq-three-parts} we conclude that
\[
\begin{split}
 \int_{{U_{\eps,T}}} \abs{\nabla u_j-\nabla u_k}^p \ud x \ud t&\leq
 c(\abs{I_2}+\abs{I_3}) +\int_{U} (u_j-u_k)^2(x,\eps) \eta(x) \ud x.
 \end{split}
\]
The right hand side can be made arbitrary small by choosing $j$ and
$k$ large enough, and a suitable $\eps$ since the last integral
converges for almost every $\eps$ (the choice of $\eps$ can be done
independent of $j,k$).  This shows that $(\nabla u_i)$ is a Cauchy
sequence in $L^p(U_{\eps,T})$, and thus it converges. We can choose a
subsequence such that $\nabla u_i\to\nabla u$ a.e.\ in $U_{\eps,T}$ as
$i\to\infty$, and further by diagonalizing with respect to $\eps$ we
can pass to subsequence such that $\nabla u_i\to\nabla u$ a.e.\ in
$U_T$ as $i\to\infty$. This and the uniform estimate in Theorem
\ref{thm:gradient-est} implies that 
\[
\begin{split}
\nabla u_i\to\nabla u \qquad \text{in} \quad L^q(U_T) 
\end{split}
\]  
as $i\to\infty$, for $q<p-1+1/(n+1)$. To see this, pick $\tilde{q}$
such that $q<\tilde{q}<p-1+1/(n+1)$. By the pointwise convergence of
the gradients established above, we also have convergence in measure. 
For any $\lambda>0$, we have
\begin{align*}
  \int_{U_T}&\abs{\nabla(u-u_j)}^q\,dz \\
  \leq & c\int_{\{\abs{\nabla(u-u_j)}<
    \lambda\}}\abs{\nabla(u-u_j)}^q\,dz+
  c\int_{\{\abs{\nabla(u-u_j)}\geq
    \lambda\}}\abs{\nabla(u-u_j)}^q\,dz\\
  \leq & c\abs{U_T}\lambda^q+c\abs{\{\abs{\nabla(u-u_j)}\geq
    \lambda\}}^{1-q/\tilde{q}}
  \left(\int_{U_T}\abs{\nabla(u-u_j)}^{\tilde{q}}\,dz\right)^{q/\tilde{q}}.
\end{align*}
by H\"older's inequality. By the convergence in measure and the bound
in $L^{\tilde{q}}$, we get that
\begin{displaymath}
  \limsup_{j\to \infty}\norm{\nabla(u-u_j)}_{L^q(U_T)}\leq c\lambda.
\end{displaymath}
Since $\lambda$ was arbitrary, the claim follows from this.
\end{proof}

\begin{theorem}
  Let $(p_i)$, $(u_i)$, $q$, $U$, and $u$ be as in
  Theorem~\ref{thm:weak-conv}.  Then there is a subsequence
  $(u_{i})$and a function $u$ such that
  \begin{displaymath}
    u_{i}\to u \quad \text{in }L^q(0,T;W^{1,q}(U)).
  \end{displaymath}
  The function $u$ is a solution to the Cauchy problem
  \begin{displaymath}
    \begin{cases}
      \pdo{u}{t}-\mathrm{div}(\abs{\nabla u}^{p-2}\nabla u)=0,& \text{in
      }\R^n\times (0,T),\\
      u(\cdot,0)=\nu. &
    \end{cases}
  \end{displaymath}
  
  If the initial trace $\nu$ is such that the solution to the limiting
  Cauchy problem is unique, then the whole sequence converges to the
  unique solution.
\end{theorem}
\begin{remark}\label{rem:uniqueness}
  The uniqueness of solutions to the Cauchy problem is open for
  general initial measures $\nu$. However, uniqueness holds if $\nu$
  is an $L^1$ function, see \cite{dibenedettoh89}. Further, uniqueness
  is known for the special case of Dirac's delta, and then the unique
  solution is the Barenblatt solution, see \cite{kaminvaz88}. Hence
  \eqref{eq:barenblatt-conv} follows from this theorem.
\end{remark}

\begin{proof}
  As the first step, we show that $u$ is a weak solution in
  $\R^n\times (0,T)$. Pick a test function $\varphi\in
  C_0^\infty(\R^n\times (0,T))$, and choose the space-time cylinder
  $U_T$ so that the support of $\varphi$ is contained in $U_T$.  The
  sequence $(\abs{\nabla u_i}^{p_i-2}\nabla u_i)$ is bounded in
  $L^r(U_T)$ for some $r>1$, and thus converges weakly in $L^r(U_T)$,
  up to a subsequence, to some limit. By the pointwise convergence of
  the gradients established in Theorem~\ref{thm-main} and the fact
  that $p_i\to p$, the weak limit must be $\abs{\nabla u}^{p-2}\nabla
  u$. This follows from an estimate similar to \eqref{eq:any_z} in
  the proof of Theorem~\ref{thm-main}. Hence we have
  \begin{align*}
    \int_{U_T}&-u\frac{\partial\varphi}{\partial t}+\abs{\nabla
      u}^{p-2}\nabla u\cdot \nabla \varphi\,dx\,dt\\
    &=\lim_{i\to \infty}\left(\int_{U_T}-u_i\frac{\partial\varphi}{\partial t}+\abs{\nabla
      u_i}^{p_i-2}\nabla u_i\cdot \nabla \varphi\,dx\,dt\right)=0.
  \end{align*}
  Since $\varphi$ was arbitrary, $u$ is a weak solution.

  Next we show that $u$ takes the right initial values in the sense of
  distributions. Define a linear approximation of the characteristic
  function of the interval $[t_1,t_2]$ as
\[
\chi_{t_1,t_2}^{h,k}(t)=
\begin{cases}
  0,& t\le t_1-h\\
  (t+h-t_1)/h,& t_1-h<t<t_1\\
  1,& t_1<t<t_2\\
  (t_2+k-t)/k,& t_2<t<t_2+k\\
 0, &t\geq t_2+k,
\end{cases}
\]
where $0\le t_1-h$.
Let $\vp \in
C^{\infty}_0(\Rn)$ and choose $\vp(x) \chi_{t_1,t_2}^{h,k}(t)$ as a test
function in the weak formulation. We obtain
\beq
\label{eq-before-limits}
\begin{split}
&\abs{\aveint{t_2}{t_2+k}\int_{\Rn } u_i \vp \, d x \,dt -
  \aveint{t_1-h}{t_1}\int_{\Rn}
u_i \vp \, d x \,d t}  \\
&\hspace{5 em} \le \abs{\int_{t_1-h}^{t_2+k}\int_{\Rn}
  {\abs{\nabla u_i}^{p-2}\nabla u_i \chi_{t_1,t_2}^{h,k}\cdot  \nabla\vp \   d x\,dt}}.
\end{split}
\eeq
Next we pass to limits in a particular order. 
Since
$t \mapsto u_i(\cdot,t)$ is a continuous function
having values in $L^2$,
$u_i \in C((0,T), L^2_{\trm{loc}}(\Rn))$, see for example page 106 in \cite{showalter97} along with similar estimates as in the proof of Theorem \ref{thm:weak-conv},  it follows that
\[
\aveint{t_1-h}{t_1}\int_{\Rn} u_i \vp \, d x \,d
t\to \int_{\Rn} u_i(x,t_1) \vp(x) \, d x (x),
\]
as $h\to 0$. Furthermore, the initial condition implies
\[
\int_{\Rn} u_i(x,t_1) \vp(x) \, d  x (x) \to \int_{\Rn}  \vp(x)
\, d \nu (x),
\]
as $t_1\to 0$. As then $i\to \infty$, we obtain
\[
\aveint{t_2}{t_2+k}\int_{\Rn } u_i \vp \, d x 
\,dt \to \aveint{t_2}{t_2+k}\int_{\Rn } u \vp \,
d x \,dt
\]
due to Theorem \ref{thm:weak-conv}. Finally,
by passing to zero with $k$, it follows that
\[
\aveint{t_2}{t_2+k}\int_{\Rn } u  \vp \, d x \,dt
\to \int_{\Rn} u(x,t_2) \vp(x) \, d x (x),
\]
since the weak solution $u$ belongs to
$C((0,T),L^2_{\trm{loc}}(\Rn))$. Observe that we only use the
continuity on an open interval $(0,T)$.

Consider next the right hand side of \eqref{eq-before-limits}. Let
first $h\to 0$ and $t_1\to 0$ in this order. The uniform integrability estimate in Theorem \ref{lem:gradient-est} implies
\[
\begin{split} 
\Big|\int_{0}^{t_2+k}\int_{\Rn} &\abs{\nabla u_i}^{p-2}\nabla u_i \cdot(\chi_{0,t_2}^{0,k}\nabla \vp) \, d x \,dt\Big|\\
& \leq 
\int_{0}^{t_2+k}\int_{\spt \vp} \abs{\nabla u_i}^{p-1} \,d x \,dt\\
& \leq C (\abs{t_2+k}\abs{U})^{(q-p+1)/q} \Big(\int_{0}^{t_2+k}\int_{\spt \vp} \abs{\nabla u_i}^{q} \,d x \,dt\Big)^{(p-1)/q}.
\end{split}
\]
Then we pass to limits
with $i$ and $k$, merge the estimates, and obtain
\beq
\label{grad2}
 \abs{\int_{\Rn} u(x,t_2) \vp(x) \, d x (x)-\int_{\Rn}
\vp(x) \,d \nu(x)}\leq C t_2^{(q-p+1)/q} .
\eeq
 The right hand
side tends to zero, as $t_2\to 0$, and we have shown
that $u$ takes the right initial values. 

It remains to prove the claim about the situation when we have
uniqueness. We argue by contradiction. Let $u$ be the unique solution
to the limit problem. If the whole sequence does not converge to $u$,
there are indices $i_k$, $k=1,2,3,\ldots$, and a number $\eps>0$ such
that
\begin{displaymath}
  \lim_{k\to \infty}\norm{u-u_{i_k}}_{L^q(0,T;W^{1,q}(U))}\geq \eps.
\end{displaymath}
The first part of the theorem applies to this subsequence, and we get
\begin{displaymath}
  \liminf_{k\to \infty}\norm{u-u_{i_k}}_{L^q(0,T;W^{1,q}(U))}=0
\end{displaymath}
by uniqueness, which is a contradiction. The proof is complete.
\end{proof}

\def\cprime{$'$} \def\cprime{$'$}


\begin{thebibliography}{BDGO97}

\bibitem[BDGO97]{boccardodgo97}
L.~Boccardo, A.~Dall'Aglio, T.~Gallou{\"e}t, and L.~Orsina.
\newblock Nonlinear parabolic equations with measure data.
\newblock {\em J. Funct. Anal.}, 147(1):237--258, 1997.

\bibitem[DH89]{dibenedettoh89}
E.~DiBenedetto and M.~A. Herrero.
\newblock On the {C}auchy problem and initial traces for a degenerate parabolic
  equation.
\newblock {\em Trans. Amer. Math. Soc.}, 314(1):187--224, 1989.

\bibitem[DiB93]{dibenedetto93}
E.~DiBenedetto.
\newblock {\em Degenerate parabolic equations}.
\newblock Universitext. Springer-Verlag, New York, 1993.


\bibitem[FHKM]{fujishimahkm} Y.~Fujishima, J.~Habermann, J.~Kinnunen
  and M.~Masson \newblock Stability for parabolic quasiminimizers.
  \newblock Submitted. Available at
  \url{https://www.mittag-leffler.se/preprints/files/IML-1314f-01.pdf}.
  

\bibitem[KV88]{kaminvaz88}
S.~Kamin and J.L.~V\'azquez.
\newblock Fundamental solutions and asymptotic behaviour for the
$p$-Laplacian equation
\newblock {\em Rev. Mat. Iberoam.}, 4(2):339--354, 1988


\bibitem[KL00]{kinnunenl00}
J.~Kinnunen and J.~Lewis.
\newblock Higher integrability for parabolic systems of $p$-Laplacian type
\newblock {\em Duke Math J.}, 102(2):253--272, 2000


\bibitem[KL05]{kinnunenl05}
J.~Kinnunen and P.~Lindqvist.
\newblock Summability of semicontinuous supersolutions to a quasilinear
  parabolic equation.
\newblock {\em Ann. Sc. Norm. Super. Pisa Cl. Sci. (5)}, 4(1):59--78, 2005.



\bibitem[KL06]{kinnunenl06}
J.~Kinnunen and P.~Lindqvist.
\newblock Pointwise behaviour of semicontinuous supersolutions to a quasilinear
  parabolic equation.
\newblock {\em Ann. Mat. Pura Appl. (4)}, 185(3):411--435, 2006.

\bibitem[KP09]{kuusip09}
T.~Kuusi and M.~Parviainen.
\newblock Existence for a degenerate {C}auchy problem.
\newblock {\em Manuscripta Math.}, 128(2):213--249, 2009.

\bibitem[KP10]{kinnunenp10}
J.~Kinnunen and M.~Parviainen.
\newblock Stability for degenerate parabolic equations.
\newblock {\em Adv. Calc. Var.}, 3(1):29--48, 2010.

\bibitem[LM98]{LiMartio}
G.~Li and O.~Martio.
\newblock Stability of solutions of varying degenerate elliptic equations.
\newblock {\em Indiana Univ. Math. J.}, 47(3):873--891, 1998.

\bibitem[Lind87]{LqvistStab}
P.~Lindqvist.
\newblock Stability for the solutions of {${\rm div}\,(\vert \nabla u\vert
  ^{p-2}\nabla u)=f$} with varying {$p$}.
\newblock {\em J. Math. Anal. Appl.}, 127(1):93--102, 1987.

\bibitem[Lind93]{LqvistRayleigh}
P.~Lindqvist.
\newblock On nonlinear {R}ayleigh quotients.
\newblock {\em Potential Anal.}, 2(3):199--218, 1993.


\bibitem[Luk]{lukkari}
T.~Lukkari.
\newblock Stability of solutions to nonlinear diffusion equations.
\newblock Submitted. Available at \url{http://arxiv.org/abs/1206.2492}.

\bibitem[Nau84]{naumann84}
J.~Naumann.
\newblock {\em Einf\"uhrung in die {T}heorie parabolischer
  {V}ariationsungleichungen}, volume~64 of {\em Teubner-Texte zur Mathematik}.
\newblock BSB B. G. Teubner Verlagsgesellschaft, Leipzig, 1984.

\bibitem[Sho97]{showalter97}
R.~E. Showalter.
\newblock {\em Monotone operators in {B}anach space and nonlinear partial
  differential equations}, volume~49 of {\em Mathematical Surveys and
  Monographs}.
\newblock American Mathematical Society, Providence, RI, 1997.

\bibitem[Sim87]{simon87}
J.~Simon.
\newblock Compact sets in the space {$L\sp p(0,T;B)$}.
\newblock {\em Ann. Mat. Pura Appl. (4)}, 146:65--96, 1987.

\bibitem[Vaz06]{VB}
J-L.~Vazquez.
\newblock {\em The Porous Medium Equation -- Mathematical Theory}.
\newblock Oxford University Press, Oxford, 2007.


\end{thebibliography}
\end{document}